\documentclass[a4paper,10pt,reqno]{amsart}
\usepackage{amsfonts,amsmath,amssymb,amsthm,amsaddr,datetime}
\usepackage{graphicx}
\usepackage{fixmath}
\usepackage[numbers,sort&compress]{natbib}  % ĚíĽÓŐâĐĐ
\usepackage[hidelinks]{hyperref}
\usepackage{tikz}
\usepackage{mathrsfs}
\usepackage{mathtools}
\usepackage{stmaryrd}
\usepackage{psfrag}
\usepackage{epsf}
\usepackage{url,rotating}
\usepackage[cp1250]{inputenc}
\usepackage{latexsym}
\usepackage{setspace}
\usepackage{enumitem}
\usepackage{multirow}
\usepackage[margin=1.2in]{geometry}
\usepackage{makecell}

\usepackage[none]{hyphenat} % ˝űÖąÁ¬×Ö·ű¶Ď´Ę
\usepackage{microtype}      % ¸ÄÉĆÎ˘ĹĹ°ćŁ¬ĽőÉŮÁ¬×Ö·űĘąÓĂ
\usepackage{ragged2e}       % Ěáą©¸üşĂµÄ¶ÔĆëżŘÖĆ

\newtheorem{thm}{Theorem}

\newtheorem{lemma}{Lemma}

\newtheorem{corollary}[thm]{Corollary}

\newtheorem{pb}{Problem}[section]

\newenvironment{kst}
{\setlength{\leftmargini}{2\parindent}
 \begin{itemize}
 \setlength{\itemsep}{-1.1mm}}
{\end{itemize}}

\begin{document}

\pagestyle{plain}
\setlength{\baselineskip}{15pt}
\title{Regular graphs are universally 3-edge-weightable}

\author{Kecai Deng}

\address{Fujian Province University Key Laboratory of Computational Science,\\
School of Mathematical Sciences, Huaqiao University, \\
             Quanzhou 362000, Fujian, P.R. China
}

\email{\qquad  kecaideng@126.com {\bf (K. Deng)}}
\thanks{Kecai Deng is the corresponding author. \\
\hspace*{4.3mm}Kecai Deng is partially supported by NSFC (No. 11701195), and Fundamental Research Funds for the Central Universities (No. ZQN-904).}

\makeatletter
\@namedef{subjclassname@2020}{\textup{2020} Mathematics Subject Classification}
\makeatother

\begin{abstract}
A graph is universally $k$-edge-weightable if for every $k$-element set $Q\subset\mathbb{R}$, it admits a proper $Q$-edge weighting. The settled 1-2-3 conjecture implies that for any arithmetic progression $\{a,b,c\}$, every nice regular graph has a proper $\{a,b,c\}$-edge weighting. We prove that this remains valid for all 3-element set $\{a,b,c\}$ with $c-b \neq b-a$. Consequently, every nice regular graph is universally $3$-edge-weightable.
\end{abstract}

\maketitle

\section{Introduction}
All graphs considered are simple, finite, and undirected. For notation or terminology not defined here, we refer to \cite{Bondy}. A graph is \emph{nice} if it contains no $K_2$ component. For $G=(V,E)$ and $Q\subset\mathbb{R}$, a \emph{$Q$-edge weighting} is a mapping $w:E\to Q$. When $Q=\{1,2,\ldots,k\}$, it is called a \emph{$k$-edge weighting}. The \emph{weighted degree} of $v$ is $d_w(v)=\sum_{e\ni v}w(e)$. An edge $uv$ is a \emph{conflict} if $d_w(u)=d_w(v)$; $w$ is \emph{proper} if it has no conflicts.

The \emph{1-2-3 conjecture}, posed by Karo\'nski, {\L}uczak, and Thomason~\cite{Thomason}, asserts that every nice graph has a proper $\{1,2,3\}$-edge weighting. This conjecture is a local version of the classical--and still active--irregularity strength problem \cite{Chartrand,Faudree,Przybylo6,Przybylo7,Kalkowski2011,AlonWei,MaXie}. Research on the 1-2-3 conjecture has progressed along several fronts. Upper bounds for general graphs were successively lowered: from 30 \cite{AddarioBerry}, to 16 \cite{AddarioBerry1}, 13 \cite{Wang}, 5 \cite{Kalkowski}, and finally to 3 \cite{Keusch1} after an intermediate improvement to 4 \cite{Keusch}. For specific classes, it was verified for dense graphs \cite{Zhong}, $d$-regular graphs with large $d$ \cite{Przybylo}, graphs with $\delta=\Omega(\log\Delta)$ \cite{Przybylo1}, and shown to hold asymptotically almost surely for random graphs \cite{AddarioBerry1}. Complexity results include NP-completeness for deciding proper $\{1,2\}$-edge weightings \cite{Dudek} and a polynomial-time algorithm for bipartite graphs \cite{Thomassen1}. The conjecture was affirmatively resolved by Keusch \cite{Keusch1}, leading to the 1-2-3 theorem. For more 1-2-3 conjecture related problems, refer to \cite{Bensmail20231,Kalkowski2017,Grytczuk,Baudona,Przybylo2,Wong1,dq1} and the references cited there.

Various strengthenings of the 1-2-3 conjecture have been studied. The list version \cite{Bartnicki09} asks whether every nice graph is \emph{edge-weight $3$-choosable}, i.e. for any assignment of $3$-element lists $L(e)$ to edges, there exists a proper edge weighting $w$ with $w(e)\in L(e)$ for all $e$. Wong and Zhu~\cite{Wong} conjectured a stronger property called \emph{$(1,3)$-choosability}, where each vertex receives a fixed weight (from a $1$-element list) and each edge receives a weight from a $3$-element list; the weighted degree of a vertex is then $d_w(v)=w(v)+\sum_{e\ni v}w(e)$. These properties form a clear hierarchy:
\[
(1,3)\text{-choosable}\Rightarrow\text{edge-weight 3-choosable}\Rightarrow\text{the 1-2-3 theorem}.
\]

Progress on the strongest property, $(1,3)$-choosability, has been steady: Cao~\cite{Cao} proved that every nice graph is $(1,17)$-choosable, improved by Zhu~\cite{Zhu} to $(1,5)$-choosability. Consequently, every nice graph is edge-weight $5$-choosable. However, it remains open whether every nice graph is $(1,3)$-choosable, or even edge-weight $3$-choosable. Partial results show that several graph classes are $(1,3)$-choosable, including complete graphs and complete bipartite graphs~\cite{Bartnicki09}, generalized Halin graphs~\cite{Liang2}, graphs with maximum average degree less than $\frac{11}{4}$~\cite{Liang}, dense Eulerian graphs~\cite{Luzhu}, and Cartesian products of paths and cycles~\cite{Tang2025}.

The 1-2-3 theorem guarantees a proper $\{1,2,3\}$-edge weighting for every nice graph.
A natural weakening of the unsolved  $(1,3)$-choosability conjecture   and edge-weight $3$-choosability conjecture, is to ask whether this extends to  any  three-element set $Q\subset\mathbb{R}$, not just $\{1,2,3\}$.
We call a graph \emph{universally $3$-edge-weightable} if it admits a proper $Q$-edge weighting for every such $Q$.

For regular graphs, affine transformation from the 1-2-3 theorem immediately gives a proper $\{a,b,c\}$-edge weighting for any arithmetic progression $\{a,b,c\}$ (where $b-a=c-b$).
The non-arithmetic case, however,  requires genuinely new constructions.
This paper provides the first such construction, which implies that every regular graph is universally $3$-edge-weightable.

\begin{thm}\label{thm1}
For any positive real numbers $d_1,d_2$ with $0<d_1<d_2$, every $k$-regular graph with $k\ge 3$ admits  a proper $\{-d_1,0,d_2\}$.
\end{thm}

\begin{corollary}\label{thm2}
Every $k$-regular graph $(k\ge 2)$ is universally $3$-edge-weightable.
\end{corollary}
\begin{proof}
For $k=2$, a $2$-regular graph consists of disjoint cycles.
Graphs with maximum average degree less than $\frac{11}{4}$ are known to be $(1,3)$-choosable~\cite{Liang}, implying the existence of a proper $\{a,b,c\}$-edge weighting.
Thus we may assume $k\ge 3$.

It suffices to show that $G$ admits a proper $\{-(b-a),0,c-b\}$-edge weighting $w'$. Indeed, if we define $w=w'+b$ (that is, we add $b$ to every edge weight), the regularity of $G$ ensures that $w$ is a proper $\{a,b,c\}$-edge weighting.

Consider three cases. If $b-a<c-b$, Theorem~\ref{thm1} directly provides a proper $\{-(b-a),0,c-b\}$-edge weighting. If $b-a=c-b$, Keusch's theorem~\cite{Keusch1} gives a proper $\{1,2,3\}$-edge weighting $w''$ of $G$. Setting $w'=(b-a)(w''-2)$ yields a proper $\{-(b-a),0,c-b\}$-edge weighting. Finally, if $b-a>c-b$, Theorem~\ref{thm1} supplies a proper $\{-(c-b),0,b-a\}$-edge weighting $w'''$; then $w'=-w'''$ is the desired weighting. In all cases the required weighting exists, proving the corollary.
\end{proof}

This settles the regular case completely. Whether universal $3$-edge-weightability holds for all nice graphs remains open:

\begin{pb}[Universal 3-edge-weightability]
Is every nice graph universally $3$-edge-weightable?
\end{pb}

Theorem~\ref{thm1} is proved in Section~2.

\section{Proof of Theorem \ref{thm1}}\label{sec:proof}

We will use the following three lemmas.

\begin{lemma} \label{lemma1} \cite{Bondy} Each $i$-degenerate graph $(i\geq1)$ is $(i+1)$-colorable.\end{lemma}

\begin{lemma} \label{lemma2} \cite{dq}  Let $G=(V,E)$ be a graph, $I_0$ be a maximum independent set of $G$, and $I_1$ be an independent set of $G[V\setminus I_0]$. Then the bipartite graph $G[I_0,I_1]$ admits a matching   saturating $I_1$.\end{lemma}

\begin{lemma} \label{lemma3}   Let $G=(I_0\cup I_1,E)$ be a  bipartite graph with $d(v)\geq3$ for each $v\in I_1$; let   $d_1$ and $d_2$ be two real numbers with $0<d_1<d_2$.  Then $G$ admits a proper $\{-d_1,0,d_2\}$-edge weighting.\end{lemma}

\begin{proof} We construct an edge weighting $w$
 such that every vertex in   $I_1$ attains a weighted degree in $\{-3d_1,d_2\}$ and each one in $I_0$ has weighted degree in $(-\infty,d_2-d_1)\setminus \{-3d_1\}$.
	
	Initially, we assign weight 0 to every edge. 	Let $I_1=\{x_i:i=1,2,\ldots,p\}$.

	First, label three of the edges incident to $x_1$ with $-d_1$. Then $d_w(x_1)=-3d_1$, and for each $y\in I_0$, $d_w(y)\in\{0,-d_1\}\subseteq (-\infty,d_2-d_1)\setminus \{-3d_1\}$.
	
	Second, for $i=2,3,\ldots,p$, suppose:	
	
	(i) for each  $j\geq i$, no edge incident to $x_j$ is weighted;
	
	(ii) for each vertex $y$ in $I_0$, $d_{w}(y)\in (-\infty,d_2-d_1)\setminus \{-3d_1\}$;
	
	(iii) for each $j<i$, $d_w(x_j)\in\{-3d_1,d_2\}$.

 	\emph{Case 1:} If $x_i$ has a neighbor $y^\ast$ with $d_w(y^\ast)\in(-\infty,-d_1)\setminus\{-d_2-3d_1\}$,
let $w(x_iy^\ast)=d_2$.
After this assignment, $d_w(x_i)=d_2$,   $d_w(y^\ast) \in (-\infty,d_2-d_1)\setminus\{-3d_1\}$, and each $y\in I_0\setminus\{y^\ast\}$ still satisfies $d_w(y) \in (-\infty,d_2-d_1)\setminus\{-3d_1\}$.
Thus, conditions (i)-(iii) are maintained for $i+1$.
	
	\emph{Case 2:} If each neighbor of $x_i$    has a   weighted degree in
	$[-d_1,d_2-d_1)\cup\{-d_2-3d_1\}$, let $y_1,y_2$ and $y_3$ be three of them. Set $w(x_iy_1)=w(x_iy_2)=w(x_iy_3)=-d_1$. Then $d_w(x_i)=-3d_1$.  For each $j=1,2,3$,  $d_w(y_j)\in [-2d_1,d_2-2d_1)\cup\{-d_2-4d_1\}$, which is a subset of    $(-\infty,d_2-d_1)\setminus \{-3d_1\}$. What is more, each $y\in I_0\setminus\{y_1,y_2,y_3\}$ still satisfies $d_w(y) \in (-\infty,d_2-d_1)\setminus\{-3d_1\}$.  Thus, conditions (i)-(iii) hold  for $i+1$.

Consequently, each vertex in  $I_0$ attains a weighted degree in $(-\infty,d_2-d_1)\setminus \{-3d_1\}$, while each one in  $I_1$ has weighted degree in $\{-3d_1,d_2\}$, and  we obtain a proper $\{-d_1,0,d_2\}$-edge weighting.	
	\end{proof}

\begin{proof}[Proof of Theorem \ref{thm1}]
Let $G=(V,E)$ be a $k$-regular graph with $k\geq3$.

\noindent\textbf{Partition.}
We partition $V$ into independent sets as follows. Choose a maximum independent set $I_0$ in $G$. For $i\geq 1$, let $I_i$ be a maximum independent set in $G_i$, where $G_{i}\triangleq G[V\setminus (\cup_{j< i}I_j)]$. Let $\ell$ be the maximum integer such that $I_{\ell}\neq \emptyset$. Then $I_{\ell}$ is an independent set, and $1\leq \ell\leq k+1$ since $G$ is nonempty and each vertex in $I_\ell$ has a neighbor in $I_i$ for each $i<\ell$. If $\ell=1$, the conclusion holds from Lemma \ref{lemma3}. So suppose $2\leq \ell\leq k+1$.

We will construct a proper $\{-d_1,0,d_2\}$-edge weighting $w$   in two phases. In the weighting process, for $i=1,2,\ldots,\ell,$ a vertex $v$ in $I_i$ is called:
 \begin{kst} \item \vspace{1mm}  \emph{Type I} if $d_w(v)=(i-1)d_2$, and $v$ is incident to exactly $i-1$ edges of weight $d_2$  connecting $v$ to $\cup_{j\geq 1}I_j$;
 	\item\vspace{1mm}	   \emph{Type II} if $d_w(v)=id_2-d_1$,  and $v$ is incident to exactly $i$ edges of weight $d_2$ connecting $v$ to $\cup_{j\geq 1}I_j$ and one  edge of weight $-d_1$ connecting $v$ to $I_0$;
 	\item \vspace{1mm}  \emph{Type III} if $v\in I_i$ for some $i\geq2$, $d_w(v)=id_2-2d_1$, and $v$ is incident to exactly $i$ edges of weight $d_2$ connecting $v$ to $\cup_{j\geq 1}I_j$ and two edges of weight $-d_1$: one connecting $v$ to $I_0$ and the other connecting $v$ to a vertex in $I_1$;
 	\item \vspace{1mm}  \emph{hungry} if $v\in I_i$ for some $i\geq2$, $d_w(v)=pd_2$ for some $p\leq i-2$  and $v$ is incident to exactly $p$ edges of weight $d_2$  connecting $v$ to $\cup_{j\geq 1}I_j$.
 \end{kst}

\noindent\textbf{Goal.}
We will make sure $w$ finally satisfies the following conditions:
\begin{kst}
\item[(1)]\vspace{1mm}  for  each vertex $v$  in $I_i$ with $i\geq2$, $d_w(v)\in \{(i-1)d_2,id_2-2d_1,id_2-d_1\}$ (that is, $v$ is Type I, II or III);
\item[(2)]\vspace{1mm} for each $x$ in $I_1$, $d_w(x)\in\{-3d_1,-d_1,0,d_2-d_1\}$;
\item[(3)]\vspace{1mm} for each $y$ in $I_0$, $d_w(y)\in(-\infty,d_2-d_1)$;
\item[(4)]\vspace{1mm} for each vertex $x$ in $I_1$ with   $d_w(x)=-3d_1$, $x$ has at least two neighbors in $I_0$, and  every neighbor $y$ of  $x$ in $I_0$ satisfies  $d_w(y)\in(-\infty,d_2-d_1)\setminus\{-3d_1\}$;
\item[(5)]\vspace{1mm} for each vertex $x$ in $I_1$ with   $d_w(x)=-d_1$, $x$ has a unique   neighbor $y_x$ in $I_0$, and $d_w(y_x)=0$;
\item[(6)]\vspace{1mm} for each vertex $x$ in $I_1$ with   $d_w(x)=0$, every neighbor $y$ of  $x$ in $I_0$ satisfies  $d_w(y)\in(-\infty,d_2-d_1)\setminus\{0\}$.
\end{kst}
Note that in (1),  $\max\{(i-1)d_2,id_2-2d_1\}<id_2-d_1$ because $d_1<d_2$. Moreover, for $1\leq i<j$, one has $id_2-d_1<\min\{(j-1)d_2,jd_2-2d_1\}$.  So (1)-(3) ensure  there is no conflict in $G[\cup_{i\geq1}I_i]$, and there is  no   conflict between $I_0$ and $\cup_{i\geq2}I_i$.   Clearly,   (3)-(6) ensure  there is no conflict between $I_0$ and $I_1$.

The construction proceeds in two phases. Phase~1 ensures condition (1)-(3) hold; Phase~2 ensures conditions (4)-(6) hold, with (1)-(3) remaining valid.

In the beginning, every edge is weighted with 0.

\subsection*{Phase 1: Initial weighting}

 In this phase, we will make sure:
 \begin{kst} \item[(1')]\vspace{1mm}
 	 every vertex in $\cup_{i\geq1}I_i$  is Type I or II;
 \item[(2')]\vspace{1mm} for each $y$ in $I_0$, $d_w(y)\in(-\infty,0]$;
  \item[(3')]\vspace{1mm} no Type I vertex in $I_1$   has a      Type I neighbor in  $\cup_{i\geq 2}I_i$.
 \end{kst}

Clearly, the validity of conditions (1')-(3') guarantees that (1)-(3) hold.

\emph{Step 1}  By Lemma~\ref{lemma2},   there exists a matching $M_{\ell-1,\ell}$ in the bipartite graph $G[I_{\ell-1},I_\ell]$ that saturates $I_l$. Weight each edge in $M_{\ell-1,\ell}$ with $d_2$.  Then when $\ell=2$, each vertex in $I_2$ is Type I (has weighted degree $d_2=(2-1)d_2$); when $\ell\geq3$, each vertex in $I_\ell$ is hungry (has weighted degree $d_2\leq (\ell-2)d_2$).  Let $I^\ast_{\ell-1}$ be the set of vertices in $I_{\ell-1}$ that are saturated by $M_{\ell-1,\ell}$.
  \begin{kst}
  \item \vspace{1mm} When $\ell=2$, by Lemma~\ref{lemma2},   there exists a matching $M^\ast_{0,1}$ in the bipartite graph $G[I_{0},I^\ast_{1}]$ that saturates $I^\ast_{1}$; weight each  edge in $M^\ast_{0,1}$ with $-d_1$, which makes each vertex in $I^\ast_{1}$ Type II (has weighted degree $d_2-d_1$), and each one in $I_1\setminus I^\ast_{1}$ Type I (has weighted degree $0=(1-1)d_2$). Then (1') and (2')   already hold in this case.

  \item\vspace{1mm}    When   $\ell=3$, each vertex in $I^\ast_{2}$ is Type I (has weighted degree $d_2=(2-1)d_2$),    each vertex in $I_2\setminus I^\ast_{2}$ is hungry (has weighted degree $0\leq(2-2)d_2$), and each  (hungry)  vertex $u$ in $I_3$ is incident to an edge  of weight $d_2$ whose  other endpoint lies in   $I_{2}$.

  \item \vspace{1mm}  When   $\ell\geq4$, each vertex in $I_{\ell-1}$ is  hungry (has weighted degree at most $d_2\leq [(\ell-1)-2]d_2$), and each (hungry) vertex $u$ in $I_\ell$ is incident to an edge  of weight $d_2$ whose  other endpoint lies in    $I_{\ell-1}$.

      \end{kst}

\emph{Step 2}  When $\ell\geq3$,  we process the indices $i=\ell-2,\ell-3,\ldots,1$ in decreasing order. Right before proceeding  $i$, we assume:
\begin{kst}
\item[(A)]\vspace{1mm}   no edge incident to $\cup_{1\leq j\leq i}I_j$ has been weighted;
\item[(B)]\vspace{1mm}  for each $\alpha\geq i+1$, every vertex in $I_\alpha$  is Type I, Type II or hungry;

\item[(C)]\vspace{1mm}  for $\alpha\geq i+2$, if  $I_\alpha$ has hungry vertices, then: for  every hungry vertex   $u$ in $I_\alpha$ and every $\beta=i+1,\ldots,\alpha-1$, $u$ is incident to at least one edge  of weight $d_2$ whose  other endpoint lies in $I_{\beta}$.
\end{kst}
(The starting point of this inductive hypothesis holds true after Step 1.)

Let $E_{i+1}$ be the set of edges with weight $0$ in the subgraph  $G[\cup_{j\geq i+1} I_{j}]$. If $E_{i+1}\neq \emptyset$, we process the edges in $E_{i+1}$ one by one. When considering an edge $e$, if both endpoints are hungry, we set $w(e)=d_2$; otherwise we leave $e$ unweighted. After processing all edges in $E_{i+1}$, let $U_{i+1}$ be the set of hungry vertices in $\cup_{j\geq i+1} I_{j}$. By construction,  if an edge has both endpoints in $U_{i+1}$, then it has already been weighted with $d_2$; hence the induced subgraph  $G[U_{i+1}]$ (if nonempty) has all edges weighted with $d_2$.

Take an arbitrary vertex $u\in U_{i+1}$, and suppose $u\in I_\alpha$ for some $\alpha=i+1,\ldots,\ell$. Since $u$ is hungry, $d_w(u)=jd_2$ for some $j\leq \alpha-2$.
By assumption (C), for each $\beta=i+1,\ldots,\alpha-1$, $u$ is incident to at least one edge of weight $d_2$ with other endpoint in $I_\beta$.
There are $(\alpha-1)-(i+1)+1 = \alpha-i-1$ such indices $\beta$.
Thus at most $(\alpha-2) - (\alpha-i-1) = i-1$ edges of weight $d_2$ incident to $u$ can have their other endpoints in $\cup_{j\geq \alpha+1} I_j$.  Recall that $G[U_{i+1}]$   has all edges weighted with $d_2$.   So within the subgraph $G[U_{i+1}]$, the vertex $u$ has at most $i-1$ neighbors in $U_{i+1}\cap(\cup_{j\geq \alpha+1} I_j)$. Since $I_\alpha$ is an independent set  and  $u$ is arbitrary, it follows that $G[U_{i+1}]$ is $(i-1)$-degenerate. By Lemma~\ref{lemma1}, $U_{i+1}$ can be partitioned into at most $i$ independent sets in $G[U_{i+1}]$. Each such independent set is also independent in $G$. Hence we can write $U_{i+1}=J_1\oplus\cdots\oplus J_r$ with $r\leq i$, where each $J_t$ is independent in $G$.

By Lemma~\ref{lemma2}, for each $t=1,2,\ldots,r$, there exists a matching $M'_{i,t}$ in the bipartite graph $G[I_i,J_t]$ that saturates $J_t$. For every edge $e$ in $\bigcup_{t=1}^{r}M'_{i,t}$, set $w(e)=d_2$. At this point, each vertex $v\in I_i$ is incident to at most $r$ edges of weight $d_2$ from these matchings, so $d_w(v)=jd_2$ for some $j\leq r\leq i$.

Let $\overline{I}_{i}\subseteq I_i$ be the set of vertices with weighted degree exactly $id_2$. Applying Lemma~\ref{lemma2} again, there exists a matching $\overline{M}_{0,i}$ in $G[I_0,\overline{I}_i]$ that saturates $\overline{I}_i$. For each edge $e\in\overline{M}_{0,i}$, set $w(e)=-d_1$.
Then each vertex in $\overline{I}_i$ becomes Type II, and each one in $I_i\setminus\overline{I}_i$ is either Type I or hungry.

After these assignments, the following hold:
\begin{kst}
	\item[(A')]\vspace{1mm} no edge incident to $\cup_{1\leq j\leq i-1}I_j$ has been weighted;
	\item[(B')] \vspace{1mm} for each $\alpha\geq i$, every vertex in $I_\alpha$  is Type I, Type II or hungry;

	\item[(C')]\vspace{1mm} for $\alpha\geq i+1$, if  $I_\alpha$ has hungry vertices, then: for  every hungry vertex   $u$ in $I_\alpha$ and every $\beta=i,\ldots,\alpha-1$, $u$ is incident to at least one edge  of weight $d_2$ whose  other endpoint lies in $I_\beta$.
\end{kst}
Indeed, (B') holds because: if $v\in \cup_{j\geq i+1}I_{j}$ was already Type I or II after step $i+1$, it remains Type I or II; if $u\in U_{i+1}$   received an edge of weight $d_2$ from some $M'_{i,t}$, it now  becomes Type I or remains hungry (the extra $-d_1$ edges to $I_0$ do not affect the condition for vertices in $U_{i+1}$); if $v\in I_{i}$, then either    it  received at most  $i-1$ edges of weight $d_2$ from   matchings $M'_{i,t}$'s (in which case it becomes
Type I or remains hungry),  or received exactly  $i$ edges of weight $d_2$ and exactly one edge of weight $-d_1$ from $\overline{M}_{0,i}$ (in which case it becomes
Type II).

Proceeding in this way for $i=\ell-2,\ell-3,\ldots,1$, we obtain at the end   that   every vertex in  $\cup_{\alpha\geq 1}I_\alpha$ is Type I or II. Since at the end, on  one hand, each $v$ in $I_1$ has weighted degree $d_2-d_1$ (Type II) or 0 (Type I). On the other hand, if some $u$ in $I_\alpha$ ($\alpha\geq2$) is hungry, it should  satisfies (C'):    $u$ is incident to  at least one edge of   weight $d_2$ with other endpoint in    $I_\beta$ for  each $\beta=1,\ldots,\alpha-1$ (at the end $i=1$). That is,     $d_w(u)\geq(\alpha-1)d_2$ which implies that $u$ is not hungry, a contradiction. Consequently, condition (1')  holds. Moreover, every $y$ in $I_0$  satisfies $d_w(y)\in(-\infty,0]$  at this stage and (2')  holds.

For (3'), if   there exists a Type I vertex $v$ in $I_{1}$ that  has a Type~I neighbor $u$ in some $I_i$ ($i\geq 2$), let $e_v$ and $e_u$ be the edges joining $v$ and $u$ to $I_0$, respectively  (recall that each vertex in $\cup_{j\geq1}I_j$ has a neighbor in $I_0$ since $I_0$ is a maximum independent set). Clearly, $e_v, e_u$ and $uv$ have not been weighted since $v$ and $u$ are both Type I (moreover, $v$ is   in $I_1$).   Set $w(e_v)=w(e_u)=-d_1$ and set $w(uv)=d_2$. Then both $u$ and $v$ become  Type~II  (so condition (1') remains valid),  and every vertex in $I_0$ still has weighted degree at most $0$ (so condition (2') remains valid).  Repeat this step until  no Type I vertex   in $I_{1}$  has a Type~I neighbor in $\cup_{i\geq2}I_i$, and  (3') holds finally.

\subsection*{Phase 2: Conflict resolution}

For $i=0,1$, let $I_{i,0}=\{v\in I_i:d_w(v)=0\}.$ By our construction, all edges incident to vertices in $I_{0,0}\cup I_{1,0}$ currently have weight $0$.
   After Phase~1, by (1')  and (2'), the only possible conflicts in $G$ are those belonging to the bipartite subgraph $G[I_{0,0},I_{1,0}]$. If $G[I_{0,0},I_{1,0}]$ is empty, we are done. Assume therefore that it is nonempty.

    For each $x\in I_{1,0}$ with $N(x)\cap I_{0,0}\neq\emptyset$, choose an arbitrary vertex from $N(x)\cap I_0$ and denote it by $y_x$. For $i\geq 1$, let $X_{i}=\{x\in I_{1,0}: |N(x)\cap I_0|= i\}$ and $X_{i^+}=\cup_{j\geq i}X_j$.

Now consider the bipartite subgraph $H' := G[X_1\cup X_2, \cup_{i\geq2}I_i]$. On one hand, for any $x\in X_1\cup X_2$, the vertex $x$ has at least $k-2$ Type~II neighbors in $\bigcup_{i\ge 2}I_i$. This holds because, by conditions (1') and (3'), every neighbor of $x$ either lies in $I_0$ or is a Type~II vertex in $\bigcup_{i\ge 2}I_i$.
On the other hand, consider a Type~II vertex $z\in I_i$ with $i\ge 2$. By definition, $z$ is incident to exactly $i$ edges of weight $d_2$ and one edge of weight $-d_1$. Since all edges incident to vertices in $X_1\cup X_2$ currently carry weight $0$, any edge of weight $d_2$ or $-d_1$ incident to $z$ must have its other endpoint outside $X_1\cup X_2$.
Thus $z$    can be adjacent to at most $k-2$ vertices in $X_1\cup X_2$.
Thus, in $H'$, every vertex in $X_1\cup X_2$ has degree at least $k-2$, and every vertex in $\cup_{i\geq2}I_i$ has degree at most $k-3$.   By Hall's   theorem,  there exists a matching $M$ in $H'$ that saturates $X_1\cup X_2$.    Write $M=M_1\oplus M_2$, where $M_i$ saturates $X_i$ for $i=1,2$. For each
$x\in X_1\cup X_2$, let $z_x$ be the (Type II) vertex in $\cup_{i\geq2}I_i$ such that
$xz_x\in M$.

We will resolve the conflicts involving $X_{2^+}$ and $X_1$ separately.

\noindent\textbf{1. Resolving conflicts incident to $X_{2^+}$.}

Let $X^\ast_{2^+}\subseteq X_{2^+}$ be the set of vertices that have a neighbor  in $I_{0,0}$.
If $X^\ast_{2^+}=\emptyset$, then we are done. Suppose $X^\ast_{2^+}\neq\emptyset$.

 We select a vertex $x_i$ from $X^\ast_{2^+}$ (arbitrarily for $i=1$, and from the updated $X^\ast_{2^+}$ for $i>1$) and handle it according to the cases below.
After processing $x_i$, it will be removed from $I_{1,0}$ (since its weighted degree will become  non-zero), and consequently from $X_{2^+}$ and $X^\ast_{2^+}$ (this process will gradually eliminates vertices from $X_{2^+}$ until $X^\ast_{2^+}$ becomes   empty).
We then update $I_{1,0}$, $X_{2^+}$, $X^\ast_{2^+}$  accordingly.

Before processing $x_i$, we assume:
\begin{kst}
\item[(i)]\vspace{1mm} no edge incident to unprocessed vertices in $X_{2^+}$ has been   weighted;
\item[(ii)]\vspace{1mm} for every $y\in I_0$, $d_w(y)\in (-\infty, d_2-d_1)$;
\item[(iii)]\vspace{1mm} each previously processed vertex $x_j$ ($j < i$) satisfies one of the following conditions:
 \begin{kst}
 \item \vspace{1mm} $d_w(x_j)=d_2-d_1$;
  \item \vspace{1mm}   $d_w(x_j)=-3d_1$ and every neighbor $y$ of $x_j$ in $I_0$ satisfies $d_w(y)\in(-\infty, d_2-d_1)\setminus\{-3d_1\}$.
      \end{kst}
\end{kst}
(For $i=1$, conditions (i) and (ii) hold; no vertex has been processed and so condition (iii) holds vacuously.)

We distinguish two cases.

\emph{Case 1:} If there exists  $y^\ast\in N(x_i)\cap I_0$ such that $d_w(y^\ast)\in(-\infty,-d_1)\setminus\{-d_2-3d_1\}$. Clearly, $y^\ast\neq y_{x_i}$, since by the definition of $y_{x_i}$, we have $d_w(y_{x_i})=0$ at this stage.
Set $w(x_iy^\ast)=d_2$ and $w(x_iy_{x_i})=-d_1$.
After these assignments, $d_w(x_i)=d_2-d_1$,
while $d_w(y_{x_i})=-d_1$ and   $d_w(y^\ast)\in(-\infty,d_2-d_1)\setminus\{-3d_1\}$. Moreover,  the weighted degree of each vertex in  $(I_0\setminus\{y_{x_i},y^\ast\})\cup (X_{2^+}\setminus \{x_i\})$  remains unchanged.
Thus,  conditions (i)-(iii) continue to hold for $i+1$.

\emph{Case 2:} If every vertex in $N(x_i)\cap I_0$  has weighted degree in $[-d_1,d_2-d_1)\cup \{-d_2-3d_1\}$.

\emph{Subcase 2.1:} If $x_i\in X_{3^+}$.
Let $y_1,y_2,y_3\in N(x_i)\cap I_0$ and set $w(x_iy_1)=w(x_iy_2)=w(x_iy_3)=-d_1$.
This gives $d_w(x_i)=-3d_1$, while $d_w(y_1),d_w(y_2),d_w(y_3)\in [-2d_1,d_2-2d_1)\cup \{-d_2-4d_1\}$.  Then every vertex in $N(x_i)\cap I_0$   has weighted degree in $[-2d_1,d_2-d_1)\cup \{-d_2-4d_1,-d_2-3d_1\}$, which  is a subset of   $(-\infty, d_2-d_1)\setminus\{-3d_1\}$. Moreover, the weighted degree of each vertex in  $(I_0\setminus\{y_1,y_2,y_3\})\cup (X_{2^+}\setminus \{x_i\})$  remains unchanged.     Thus, conditions (i)-(iii)   continue  to hold  for $i+1$.

\emph{Subcase 2.2:} If $x_i\in X_{2}$. Assume $N(x_i)\cap I_0=\{y_{x_i},y'\}$ and  $z_{x_i}\in I_q$ for some $q\geq 2$.  Recall that  $z_{x_i}$ is Type II and so   $d_w(z_{x_i})=qd_2-d_1$. Now set $w(x_iy_{x_i})=w(x_iy')=w(x_iz_{x_i})=-d_1$. This gives $d_w(x_i)=-3d_1$, while $d_w(y_{x_i})=-d_1$, $d_w(y')\in [-2d_1,d_2-2d_1)\cup \{-d_2-4d_1\}$ and $d_w(z_{x_i})=qd_2-2d_1$ (that is, $z_{x_i}$ becomes Type III and  so condition (1) remains valid). Then  $\{d_w(y_{x_i}),d_w(y')\}\subseteq [-2d_1,d_2-2d_1)\cup \{-d_2-4d_1\}$, which   is   a subset of   $(-\infty, d_2-d_1)\setminus\{-3d_1\}$.  Moreover, the weighted degree of each vertex in   $(I_0\setminus\{y_{x_i},y'\})\cup (X_{2^+}\setminus \{x_i\})$  remains unchanged. Thus, conditions  (i)-(iii)  continue  to hold  for $i+1$.

\vspace{1mm}\textbf{Remark on the dynamics of $I_{0,0}$:} In Case~1, it is possible that $d_w(y^\ast)$ becomes $0$ after adding $d_2$ (if originally $d_w(y^\ast) = -d_2$), placing $y^\ast$ back into $I_{0,0}$. However, by assumption (i), all edges incident to $\{x_j:j>i\}$ are still unweighted, so any new conflict involving $y^\ast$ and $x_j$ ($j>i$) will be resolved when $x_j$ is processed. Moreover, $y^\ast$ cannot conflict with already processed vertices $x_s$ ($s<i$) because those have weighted degrees either $d_2-d_1$ or $-3d_1$, both different from $0$.  For Case  2, all modified weighted degrees in $I_0$ become negative;   so  $I_0$-vertices  cannot re-enter $I_{0,0}$.
\vspace{1mm}

After handling $x_i$, update $I_{0,0}, X_{2^+}$ and $X^\ast_{2^+}$, respectively.
Note that $x_i$ and all previously processed vertices $x_j$ ($j < i$) are no longer in $X^\ast_{2^+}$ (no longer in $X_{2^+}$ either),
since they either have weighted degree $d_2-d_1$ (which differs from $0$),
or have weighted degree $-3d_1$ while all their neighbors in $I_0$ have weighted degrees different from $-3d_1$. If $X^\ast_{2^+}=\emptyset$, we are done. If $X^\ast_{2^+}\neq\emptyset$ let $x_{i+1}\in X^\ast_{2^+}$ and the assumptions (i)-(iii) hold  for $i+1$. Proceeding and update $X^\ast_{2^+}$ repeatedly, we eventually eliminate all conflicts in $G[I_0,X_{2^+}]$. Moreover, conditions (1)-(3)  remain  valid currently. Hence no conflict remains incident to $X_{2^+}$ finally.

\vspace{1mm}\textbf{Remark on the dynamics of $X^\ast_{2^+}$:} Since (in Case~1) $I_{0,0}$ may expand, $X^\ast_{2^+}$ may also expand. However, $X^\ast_{2^+}$ is included in $X_{2^+}$, and $X_{2^+}$  decreases  at each step. So the recursive process will terminate  with   $X^\ast_{2^+}=\emptyset$.

\vspace{1mm}

\noindent\textbf{2. Resolving conflicts incident to $X_1$.}

After resolving conflicts for $X_{2^+}$, we update $I_{0,0}$. Note that  $X_1$ remains unchanged.
Let $X^\ast_{1}\subseteq X_{1}$ be the set of vertices in   $X_1$ that have a neighbor in (updated) $I_{0,0}$.
If $X^\ast_1=\emptyset$, we are finished.

Assume $X^\ast_1\neq\emptyset$ and let $M^\ast_1$ be the subset of $M_1$ consisting of edges incident to vertices in $X^\ast_1$.
Assign weight $-d_1$ to every edge in $M^\ast_1$.  After the assignment, for each  $x\in X^\ast_1$, $d_w(x)=-d_1$, while $d_w(y_x)$ remains $0$ and $d_w(z_x)$ decreases by $d_1$,   making $z_x$ a Type III vertex.
This    implies that  condition (5) holds, and conditions (1)-(3) remain  valid. Note that, at present, both $X^\ast_{2^+}$ and $X^\ast_{1}$ are empty, and so condition (6) holds. Condition (4) holds because: vertices with $d_w(x)=-3d_1$ come from Case~2.
In Case~2.1, $x\in X_{3^+}$ has at least three neighbors in $I_0$;
in Case~2.2, $x\in X_2$ has exactly two neighbors in $I_0$.
In both cases, all neighbors $y$ of $x$ in $I_0$ satisfy $d_w(y)\neq-3d_1$ by construction. Thus, conditions (1)-(6) all hold, and $w$ is a proper $\{-d_1,0,d_2\}$-edge weighting of $G$.

 This completes the proof of Theorem~\ref{thm1}.
\end{proof}

\section{Concluding Remarks}

This paper establishes that every regular graph is universally 3-edge-weightable.  The conflict-resolution scheme introduced in this work may be applicable to other weighting and labeling problems. The two-phase approach--comprising global matching adjustments followed by local conflict resolution--provides a framework for handling parameter-dependent constructions in combinatorial graph theory.

\noindent\textbf{Future work.} Our result naturally leads to several open questions:
\begin{itemize}
    \item Does universal 3-edge-weightability extend to all nice graphs? (Problem 1.1)

    \item What other graph classes (beyond regular graphs) admit this property?
    \item Can the techniques be extended to study edge-weight choosability problems?
\end{itemize}

\noindent\textbf{Acknowledgments}

The first author is supported by NSFC (No. 11701195) and Fundamental Research Funds for the Central Universities (No. ZQN-904).


\begin{thebibliography}{99}
\small \setlength{\itemsep}{.8mm}
\bibitem{Bondy} J.A. Bondy, U.S.R. Murty, {Graph Theory}, Graduate Texts in Mathematics, vol. 244, Springer, New York, 2008.

\bibitem{Thomason} M. Karo\'nski, T. {\L}uczak, A. Thomason, Edge weights and vertex colours, J. Comb. Theory, Ser. B {91} (2004) 151-157.

       \bibitem{Chartrand} G. Chartrand, M.S. Jacobson, J. Lehel, O. Oellermann, S. Ruiz, F. Saba, Irregular networks, Congr.
Numer. 64 (1988) 197-210.

  \bibitem{Faudree}R.J. Faudree, J. Lehel, Bound on the irregularity strength of regular graphs, Colloq. Math. Soc.
J\'{a}nos Bolyai 52 (1988) 247-256.


    \bibitem{Kalkowski2011}  M. Kalkowski, M. Karo\'{n}ski, F. Pfender, A new upper bound for the irregularity strength of graphs.
SIAM J. Discrete  Math. 25(3) (2011)  1319-1321.

\bibitem{Przybylo6} J. Przyby{\l}o, F. Wei, Short proof of the asymptotic confirmation of the Faudree-Lehel conjecture,
Electron. J. Comb. 30 (4) (2023)   Paper No. 4.27, 13 pp.


\bibitem{Przybylo7} J. Przyby{\l}o, F. Wei, On the asymptotic confirmation of the Faudree-Lehel conjecture for general
graphs, Combinatorica 43 (2023) 791-826.


\bibitem{AlonWei} N. Alon, F. Wei, Irregular subgraphs, Combin. Probab. Comput. 32 (2) (2023) 269-283.


 \bibitem{MaXie}  J. Ma, S. Xie, Finding irregular subgraphs via local adjustments, J. Combin. Theory, Ser. B 174 (2025) 71-98.

     \bibitem{AddarioBerry} L. Addario-Berry, K. Dalal, C. McDiarmid, B. Reed, A. Thomason, Vertex-colouring edge weightings, Combinatorica 27 (2007) 1-12.

     \bibitem{AddarioBerry1} L. Addario-Berry, K. Dalal, B.A. Reed, Degree constrained subgraphs, Discrete Appl. Math. 156 (7)
(2008) 1168-1174.

  \bibitem{Wang} T. Wang, Q. Yu, A note on vertex-coloring 13-edge-weighting, Front. Math. China 3 (2008)
581-587.

 \bibitem{Kalkowski} M. Kalkowski, M. Karo\'{n}ski, F. Pfender, Vertex-coloring edge-weightings: towards the 1-2-3-conjecture, J. Comb. Theory, Ser. B 100 (2010) 347-349.


\bibitem{Keusch1} R. Keusch, A solution to the 1-2-3 conjecture, J. Comb. Theory, Ser. B {166} (2024) 183-202.

    \bibitem{Keusch}R. Keusch, Vertex-coloring graphs with 4-edge-weightings, Combinatorica 43 (2023) 651-658.

        \bibitem{Zhong} L. Zhong, The 1-2-3-conjecture holds for dense graphs, J. Graph Theory 90 (2018) 561-564.


\bibitem{Przybylo}   J. Przyby{\l}o, The 1-2-3 conjecture almost holds for regular graphs, J. Comb. Theory, Ser. B 147
(2021) 183-200.

\bibitem{Przybylo1} J. Przyby{\l}o, The 1-2-3 conjecture holds for graphs with large enough minimum degree, Combinatorica 42 (2022) 1487-1512.





     \bibitem{Dudek} A. Dudek, D. Wajc, On the complexity of vertex-coloring edge-weightings, Discrete Math. Theor.
Comput. Sci. 13 (3) (2011) 45-50.

\bibitem{Thomassen1} C. Thomassen, Y. Wu, C.-Q. Zhang, The 3-flow conjecture, factors modulo $k$, and the 1-2-3-conjecture, J. Comb. Theory, Ser. B 121 (2016) 308-325.

  \bibitem{Bensmail20231}    J. Bensmail, H. Hocquard, D. Lajou, \'{E}. Sopena, A proof of the multiplicative 1-2-3 conjecture, Combinatorica 43 (2023) 37-55.

  \bibitem{Kalkowski2017}  M. Kalkowski, M. Karo\'{n}ski, F. Pfender,  The 1-2-3-conjecture for hypergraphs. J. Graph Theory
85(3) (2017)  706-715.

 \bibitem{Grytczuk}  J. Grytczuk, From the 1-2-3 conjecture to the Riemann hypothesis, Eur. J. Comb. 91 (2020), Paper 103213, 10pp.

\bibitem{Baudona} O. Baudon, J. Bensmail, J. Przyby{\l}o, M. Wo\'{z}niak, On decomposing regular graphs into locally
irregular subgraphs, Eur. J. Comb. 49 (2015) 90-104.

\bibitem{Przybylo2} J. Przyby{\l}o, M. Wo\'{z}niak, On a 1-2 conjecture, Discrete Math. Theor. Comput. Sci. 12 (2010)
101-108.

 

        

 

       \bibitem{Wong1}    T.-L. Wong, X. Zhu, Every graph is (2, 3)-choosable, Combinatorica 36 (2014) 121-127.

  

   \bibitem{dq1} K. Deng, H. Qiu, Every graph is uniform-span (2,2)-choosable: Beyond the 1-2 conjecture, 	arXiv:2506.14253 [math.CO]

\bibitem{Bartnicki09} T. Bartnicki, J. Grytczuk, S. Niwczyk, Weight choosability of graphs, J. Graph Theory {60} (2009) 242-256.

\bibitem{Wong} T.-L. Wong, X. Zhu, Total weight choosability of graphs, J. Graph Theory {66} (2011) 198-212.

\bibitem{Cao} L. Cao, Total weight choosability of graphs: towards the 1-2-3-conjecture, J. Comb. Theory, Ser. B {149} (2021) 109-146.

    \bibitem{Zhu} X. Zhu, Every nice graph is (1, 5)-choosable, J. Comb. Theory, Ser. B \textbf{157} (2022) 524-551.

\bibitem{Liang2} Y.-C. Liang, T.-L. Wong, X. Zhu, Total weight choosability for Halin graphs, Electron. J. Graph Theory Appl. {9} (1) (2021) 11-24.

\bibitem{Liang} Y.-C. Liang, T.-L. Wong, X. Zhu, Graphs with maximum average degree less than $\frac{11}{4}$ are $(1, 3)$-choosable, Discrete Math. {341} (2018) 2661-2671.

\bibitem{Luzhu} H. Lu, X. Zhu, Dense Eulerian graphs are $(1, 3)$-choosable, Electron. J. Comb. {29}(2) (2022) Paper No. P2.54, 8pp.

    \bibitem{Tang2025} Y. Tang, Y. Yao, Total list weighting of Cartesian product of graphs, Discrete Appl. Math. {367} (2025) 30-39.


          \bibitem{dq} K. Deng, H. Qiu, The 1-2 conjecture holds for regular graphs, J. Comb. Theory, Ser. B {174} (2025) 207-213.

\end{thebibliography}
\end{document}